\documentclass{amsart}
\makeatletter
\@namedef{subjclassname@2020}{%
  \textup{2020} Mathematics Subject Classification}
\makeatother
\usepackage{latexsym}
\usepackage{amssymb}
\usepackage{graphicx}
\usepackage{wrapfig}
\usepackage{amsthm}
\usepackage{float}
\usepackage{epsfig}
\usepackage{multirow}
\usepackage{xcolor}
\usepackage{caption}
\usepackage[toc,page]{appendix}
\usepackage{old-arrows} 
\usepackage{caption}
\usepackage{subcaption}
\usepackage{lmodern, mathtools}
\usepackage{mathtools}
\usepackage{hyperref}
\RequirePackage{color}
\usepackage{tikz}
\usepackage{amscd,enumerate}
\usepackage{amsfonts}
\hypersetup{ colorlinks,
linkcolor=blue,
filecolor=green,
urlcolor=blue,
citecolor=blue }
\usepackage{helvet}
\usepackage{bigstrut}
\usepackage{float, caption, booktabs, cellspace}
\usepackage{calrsfs}
\DeclareMathAlphabet{\pazocal}{OMS}{zplm}{m}{n}

\makeatletter
\@addtoreset{subfigure}{framenumber}
\makeatother

\usepackage{etoolbox}
\AtBeginEnvironment{figure}{\setcounter{subfigure}{0}}

\newtheorem{lemma}{Lemma}[section]

\newtheorem{theorem}[lemma]{Theorem}
\newtheorem{proposition}[lemma]{Proposition}
\newtheorem{remark}[lemma]{Remark}

\input xygraph
\input xy
\xyoption{all}

\usepackage{array}
\newcolumntype{L}[1]{>{\raggedright\let\newline\\\arraybackslash\hspace{0pt}}m{#1}}
\newcolumntype{C}[1]{>{\centering\let\newline\\\arraybackslash\hspace{0pt}}m{#1}}
\newcolumntype{R}[1]{>{\raggedleft\let\newline\\\arraybackslash\hspace{0pt}}m{#1}}

\usepackage{multicol}

\def\l{\lambda}

\def\o{\omega}
\def\s{\sigma}

\def\remove#1{}
\newcommand{\p}{{\mathbb{P}}}

\newcommand{\C}{{\mathbb{C}}}
\newcommand{\OO}{{\mathcal{O}}}

\newcommand{\Det}{{\rm{Det}}}

\newcommand{\N}{{\mathbb{N}}}

\newcommand{\K}{{\mathbb{K}}}

\newcommand{\Sc}{{\mathbb{S}}}

\usepackage{verbatim}
\usepackage{soul}

\title{On the vanishing of the hyperdeterminant under certain symmetry conditions}

\author[E. Arrondo]{Enrique Arrondo  $^{1}$ }
\author[A. Tocino]{Alicia Tocino $^{2,\ast}$}

\address{$^{\ast}$ Corresponding author. Email: alicia.tocino@uma.es}

\address{$^{1}$ Instituto de Matem\'atica Interdisciplinar and Departamento de Álgebra, Geometría y Topología, Facultad de Ciencias Matemáticas, Universidad Complutense de Madrid, 28040 Madrid, Spain}
\email{arrondo@mat.ucm.es}

\address{$^{2}$ Departamento de Matem\'atica Aplicada, E.T.S. Ingenier\'\i a Inform\'atica, Universidad de M\'alaga, Campus de Teatinos s/n. 29071 M\'alaga,   Spain. }
\email{alicia.tocino@uma.es}


\thanks{The first author is supported by the Spanish Ministerio de Ciencia e Innovaci\'on through the project PID2021-124440NB. The second author is supported by the Junta de Andaluc\'{\i}a  through the project FQM-336 with FEDER funds and by Universidad de Málaga.
}

\usepackage[foot]{amsaddr}

\begin{document}

\begin{abstract}
Given a vector space $V$ over a field $\K$ whose characteristic is coprime with $d!$, let us decompose the vector space of multilinear forms $V^*\otimes\overset{\text(d)}{\ldots}\otimes V^*=\bigoplus _\l W_\l(X,\K)$ according to the different partitions $\l$ of $d$, i.e. the different representations of $S_d$. In this paper we first give a decomposition $W_{(d-1,1)}(V,\K)=\bigoplus_{i=1}^{d-1}W_{(d-1,1)}^i(V,\K)$. We finally prove the vanishing of the hyperdeterminant of any $F\in(\bigoplus_{\l\ne(d),(d-1,1)})\oplus W_{(d-1,1)}^i(V,\K)$. This improves the result in \cite{TocinoHyper} and \cite{repre}, where the same result was proved without this new last summand.

\end{abstract}

\maketitle

\section{Introduction}

In \cite{TocinoHyper}, the second author proved (in an implicit way, made explicit  by the first author in \cite{repre}) that the vanishing of the hyperdeterminant of a hypermatrix of size $n\times \overset{\text(d)}{\ldots}\times n$ belonging to (the direct sum of) all but two of its possible symmetries. This implies that the hypersurface defined by the hyperdeterminant contains a linear subspace of very large dimension.

In this paper, we improve that result including a subspace of one of the remaining two symmetries. The motivation is that, for matrices of size $n\times n$, in which the only two types of symmetry are symmetric matrices and skew symmetric matrices, we still have the vanishing of the determinant of skew-symmetric matrices when $n$ is odd. Hence one still could expect that hypermatrices in one of the two remaining symmetries could still have hyperdeterminant equal to zero. Namely, we could expect that, at least under certain condition on $n$, those hypermatrices that are multiplied by primitive $d$-root of unity when moved by a $d$-cycle (in a sense that we will explain) should still have hyperdeterminant equal to zero. This will be the case, although the proof will not be as simple as in the case $d=2$, but a posteriori this will show that, for $d\ge2$ we will not need any condition on $n$.

The structure of the paper is as follows. In a section of preliminaries we will recall the precise notion of symmetries of functions in $d$ variables using representation theory of the symmetric group $S_d$ and we will also recall the main properties of the hyperdeterminant. In the last section, we will prove the main results of the paper. It is striking to observe that, for the proof of our main result we will need to use again algebraic geometry, by interpreting our problem as the existence of a section of a twist of the cotangent bundle of the projective space with no zeroes.

\vspace{0.25cm}

\noindent
\textbf{Acknowledgments.} We thank Giorgio Ottaviani for his useful comments that improved the presentation of the article and led to Remark \ref{remark:subspaces}, among other enhancements.

\section{Preliminaries}

\bigskip
\textbf{Symmetries of functions.} 
Assume that $X$ is a set and $\K$ is a field. It is a standard fact that any $F\colon X\times X\rightarrow \K$ can be decomposed into a symmetric and a skew-symmetric part as
$$F(x,y)=\frac{F(x,y)+F(y,x)}{2}+\frac{F(x,y)-F(y,x)}{2}$$ 
(observe that we need the characteristic of $\K$ to be different from two).
In general, in order to have a decomposition of a function in $d$ variables as a sum of functions with different types of symmetry, one must use representation theory of the symmetric group $S_d$. Our standard reference for representation theory on $S_d$ will be \cite{fulton2013representation}, where the author assumes $\K=\C$. For an arbitrary field, one could follow \cite{Isaacs} 
(for generalities of representation theory) or \cite{repre} (for the concrete case we are dealing with, from where we conclude that the theory will work when the characteristic of $\K$ does not divide $\vert S_d\vert=d!$, which we will assume throughout the paper. 

To fix our set-up, we first observe that the symmetric group $S_d$ acts naturally on the vector space $\K^{X\times\ldots\times X}$ of functions of $d$-variables $X\times \overset{\text(d)}{\ldots}\times X\rightarrow \K$ as follows, $(\s F)(x_1,\ldots,x_d):=F(x_{\s(1)},\ldots,x_{\s(d)})$ for any $\s\in S_d$ and $F\in \K^{X\times \ldots\times X}$. 
We could restrict our attention to the action of $S_d$ on invariant vector subspaces
$C^d(X,\K)\subset \K^{X\times\ldots\times X}$. 
In fact, we are interested in the case where $X$ is a vector space $V$ and $C^d(V,\K)=V^*\otimes\overset{\text(d)}{\ldots}\otimes V^*=\{F: V\times \overset{\text(d)}{\ldots}\times V\to \K \text{ multilinear form}\}$.

We can naturally extend the action of $S_d$ to an action of the group algebra $\K[S_d]$ on $C^d(X,\K)$:  $$(\sum_{\s\in S_d}\l_\s \s F)(x_1,\ldots,x_d):=\sum_{\s\in S_d}\l_\s F(x_{\s(1)},\ldots,x_{\s(d)})$$ for any $F\in C^d(X,\K)$. 
An example of this action is given by the aforementioned decomposition for $d=2$:
$$F=\frac{(1)+(12)}{2}F+\frac{(1)-(12)}{2}F$$
where $(1)$ denotes the identity permutation. 
For a general $d$, there are as many types of symmetries as irreducible representations of $S_d$. 
Recall that every irreducible representation of $S_d$ is obtained from a partition $\l$ of $d$ and it is called $V_\l$. 

In general, following the notation of \cite[Theorem 4.5]{repre}, one has the decomposition of $C^d(X,\K)$ into different types of symmetry as follows:
$$C^d(X,\K)=\bigoplus _\l W_\l(X,\K) $$
in which each $W_\l(X,\K)$ is the sum (possibly infinite) of $V_\l$. In the specific case we are interested in, we have the following decomposition 
$$C^d(V,\K)=V^*\otimes\overset{\text(d)}{\ldots}\otimes V^*\cong \bigoplus_\l W_\l(V,\K)$$
where $W_\l(V,\K):=(\Sc_\l V)^{\dim(V_\l)}$  and $\Sc_\l V$ is the so-called Schur functor of $V$ (see \cite[Lecture 6]{fulton2013representation}, which is also a direct sum of the irreducible representation $V_\l$.

\begin{remark}\rm
Recall that if $\l=(d)$ then $W_{(d)}(X,\K)$ is the set of symmetric functions and if $\l=(1,\ldots,1)$ then $W_{(1,\ldots,1)}(X,\K)$ is the set of skew-symmetric functions. The case we are going to study most closely is when $\l=(d-1,1)$ for which $V_{(d-1,1)}$ is the so-called standard representation.
It is not straightforward to give the equations of each type of symmetry (see, for example \cite{MetropolisRota3}, \cite{RotaMetropolisStein} and \cite[Section 4]{repre}). For instance, being in the complementary of the symmetric case, that is $F\in \oplus_{\l\neq (d)}W_\l(V,\K)$, is given by a concrete equation, namely $F(v,\ldots,v)=0$ for any $v\in V$ (see \cite[Lemma 3.2]{TocinoHyper} or \cite[Lemma 4.7]{repre} for further details), which will be used in the proof of Theorem \ref{teo:main}.
\end{remark}

\noindent
Observe that the vector space of multilinear forms $F:V\times \overset{\text(d)}{\ldots}\times V\to \K$, where $V$ is a vector space with basis $\{v_1,\ldots,v_n\}$,  is isomorphic to the vector space of hypermatrices  $A=(a_{i_1\ldots i_d})$ of size $n\times \overset{\text(d)}{\ldots}\times n$ by taking $F(v_{i_1},\ldots,v_{i_d})=:a_{i_1\ldots i_d}$. So, one can talk about either symmetries of multilinear forms or symmetries of hypermatrices.

\bigskip

\textbf{Features of hyperdeterminants.}
There exists the notion of hyperdeterminant of a general hypermatrix (see \cite[Chapter 14]{Kapranov}), which is equivalent to the notion of hyperdeterminant of a multilinear form $F:V_1\times \ldots \times V_d\to \K$ for different vector spaces $V_1,\ldots,V_d$. 
Nevertheless, since we are dealing with symmetries, we focus in the case $V_1=\ldots=V_d=V$. In that event, we recall that the hyperdeterminant of $F$ is a form $\Det:V^*\otimes\overset{\text(d)}{\ldots}\otimes V^* \to\K$ (whose degree is not specially easy to compute) with the important property that $\Det(F)=0$ if and only if there exist $v_1,\ldots,v_d\in V$ such that $F(v_1,\ldots,v_{d-1},V)=F(v_1,\ldots,V,v_d)=\ldots=F(V,v_2,\ldots,v_d)=0$. 
In particular, $\Det(F)=0$ if there exists $v\in V$ such that: 
\begin{equation}\tag{$\star$} \label{estrella}
    F(v,\ldots, v, V)=\ldots=F(v, \ldots, v, V, v)=F(V,v,\ldots,v)=0.
\end{equation}

\begin{remark}\rm\label{remark:any}
    In \cite[Main Theorem]{TocinoHyper} it was proven that if $F\in W_\l(V,\K)$, for all $\l\neq (d),(d-1,1)$, then condition \eqref{estrella} is satisfied for all $v\in V$. Hence, condition \eqref{estrella} is also satisfied for $(v,F)$  when $F\in \bigoplus_{\l\neq (d), (d-1,1)} W_\l(V,\K)$ and all $v\in V$, implying that $\Det(F)=0$ (see \cite[Proposition 4.8]{repre} for an alternative proof). Observe that this result does not give any information when $d=2$ since we only have $W_{(2)}(V,\K)$ and $W_{(1,1)}(V,\K)$. Nevertheless, it is known that the determinant of antisymmetric matrices (which correspond to $F\in W_{(1,1)}(V,\K)$) of odd order is always zero. This fact leads us to expect that the hyperdeterminant will cancel out in more pieces than the ones that has already been shown. We will see in Theorem \ref{teo:main} that these new pieces will live inside $W_{(d-1,1)}(V,\K)$.
\end{remark}

\section{Main results}
In order to generalize the notion of skew-symmetric matrices glimpsed in the previous remark, we need $\K$ to have a primitive $d^{\text{th}}$-root of unity, and we call it $\o$. So, from now on, we assume that $\K$ contains a primitive $d^{\text{th}}$-root of unity.
We start with some remarks and lemmas that will be used to prove the decomposition of $W_{(d-1,1)}(V,\K)$ given in Proposition \ref{prop:descomp} and Theorem \ref{teo:main}.

\begin{remark}\label{remark:desco}\rm
    Consider $n\in \N$ and $\omega$ a  primitive  $\rm{n^{th}}$-root of unity. This implies that $\omega^n=1$ and $1+\omega+\omega^2+\ldots+\omega^{n-1}=0$.
    Observe that, for each $i\in \N, 1<i< d$,  one has the following equality:
    $$0=\omega^n-1=(\omega^{n_i}-1)((\omega^{n_i})^{r_i-1}+(\omega^{n_i})^{r_i-2}+\ldots+(\omega^{n_i})+1)$$
    \noindent
    where $n_i$ is the greatest common divisor of $n$ and $i$ and $r_i$ satisfies that $n=n_i r_i$.
    Since $\o$ is primitive, $\omega^{n_i}-1\neq 0$ so $(\omega^{n_i})^{r_i-1}+(\omega^{n_i})^{r_i-2}+\ldots+\omega^{n_i}+1=0$.
\end{remark}

\begin{lemma}\label{lemma:omega}
    Let $\o$ be a primitive  $\rm{d^{th}}$-root of unity. Then $\sum_{k=1}^{d-1}(\o^k)^i=-1$ for all $i\in\{1,\ldots,d-1\}$.
\end{lemma}

\begin{proof}
    Applying Remark \ref{remark:desco} and replacing $n$ with $i d$, $n_i$ with $i$ and $r_i$ with $d$ we obtain what we want since
    $$0=(\o^i)^d-1=(\o^i-1)(\sum_{k=1}^{d-1}(\o^i)^k+1)=(\o^i-1)(\sum_{k=1}^{d-1}(\o^k)^i+1)\Longrightarrow \sum_{k=1}^{d-1}(\o^k)^i+1=0.$$
\end{proof}

\begin{lemma}\label{lemma:suma}
    If $F\in W_{(d-1,1)}(V,\K)$, $\s$ is a $d$-cycle of $S_d$ and consider $\gamma=(1)+\s +\s^2+\ldots+\s^{d-1}$ then  $\gamma F=0$.
\end{lemma}
\begin{proof}
Since $W_{(d-1,1)}(V,\K)=\oplus V_{(d-1,1)}$, it is enough to prove that $\gamma$ acts as zero on each $V_{(d-1,1)}$. Recall that $V_{(d-1,1)}$ can be identified with the hyperplane $\{(x_1,\ldots,x_d)\in \K^d\, \colon\, x_1+\ldots+x_d=0\}$ and the action of  $\s\in S_d$ on each $(x_1,\ldots,x_d)\in V_{(d-1,1)}$ works as $\s (x_1,\ldots,x_d)=(x_{\s(1)},\ldots,x_{\s(d)})$. So, $\gamma(x_1,\ldots,x_d)=(x_1,x_2,\ldots,x_d)+(x_{\s(1)},x_{\s(2)},\ldots,x_{\s(d)})+\ldots+(x_{\s^{d-1}(1)},x_{\s^{d-1}(2)},\ldots,x_{\s^{d-1}(d)})=(0,\ldots,0)$ as we wanted. 
\end{proof}

\begin{remark}\rm
    In \cite{RotaMetropolisStein} are (originally) given the equations that describe each $F\in W_\l (V,\K)$. The previous lemma adds one additional equation that satisfies $F\in W_{(d-1,1)}(V,\K)$. When $d=3$, this new equation is the only equation that describes $W_{(2,1)}(V,\K)$, that is, $F\in W_{(2,1)}(V,\K)$ if and only if $F+(123 )F+(132)F=0$.
\end{remark}

\begin{proposition}\label{prop:descomp}
    There is a decomposition $W_{(d-1,1)}(V,\K)=\bigoplus_{i=1}^{d-1}W_{(d-1,1)}^i(V,\K)$, where $W_{(d-1,1)}^k(V,\K)=\{F\in W_{(d-1,1)}(V,\K)\,\vert\, \s F=\o^k F\}$ with $k\in\{1,\ldots,d-1\}$, $\o$ is a primitive $\rm{d^{th}}$-root of unity, $\s$ is a $d$-cycle and $d>2$.
\end{proposition}
\begin{proof}
    
First, we prove that $F\in W_{(d-1,1)}(V,\K)$ can be decomposed as $F=H_1+\ldots+H_{d-1}$.
So, consider $F\in W_{(d-1,1)}(V,\K)$ and the following sum

$$\frac{F+\omega \s F+\o^2\s^2F+\ldots+\o^{d-1}\s^{d-1}F}{d}+\frac{F+\omega^2 \s F+(\o^2)^2\s^2F+\ldots+(\o^2)^{d-1}\s^{d-1}F}{d}+$$
$$+\frac{F+\omega^3 \s F+(\o^3)^2\s^2F+\ldots+(\o^3)^{d-1}\s^{d-1}F}{d}+\ldots+\frac{F+\omega^{d-1} \s F+(\o^{d-1})^2\s^2F+\ldots+(\o^{d-1})^{d-1}\s^{d-1}F}{d}=$$
\begin{equation}\label{eq}
    =\sum_{k=1}^{d-1}\frac{F+\o^k\s F+(\o^k)^2\s^2 F+\ldots (\o^k)^{d-1}\s^{d-1}F}{d}.
\end{equation}
By grouping in the terms $F, \s F, \s^2 F,\ldots,\s^{d-1}F$ we have
$$\frac{(d-1)F}{d}+\sum_{k=1}^{d-1}\frac{\o^k }{d}\s F+\sum_{k=1}^{d-1}\frac{(\o^k)^2}{d}\s^2 F+\ldots +\sum_{k=1}^{d-1}\frac{(\o^k)^{d-1}}{d}\s^{d-1} F.$$
By Lemma \ref{lemma:omega} we have that $$\sum_{k=1}^{d-1}{\o^k }=-1, \,\sum_{k=1}^{d-1}(\o^k)^2=-1, \,\ldots, \,\sum_{k=1}^{d-1}(\o^k)^{d-1}=-1.$$
So, the above sum is as follows,
$$\frac{(d-1)F-\s F-\s^2 F-\ldots -\s^{d-1} F}{d}=F$$
which is equal to $F$ since $F=-\s F-\s^2 F-\ldots -\s^{d-1} F$ by Lemma \ref{lemma:suma}. For $k=1,\ldots,d-1$, we denote: 
$$H_{d-k}:=\frac{F+\omega^k \s F+(\o^k)^2\s^2F+\ldots+(\o^k)^{d-1}\s^{d-1}F}{d}=\frac{(1)+\omega^k \s +(\o^k)^2\s^2+\ldots+(\o^k)^{d-1}\s^{d-1}}{d}F.$$
So that, the sum in \eqref{eq} leads to $F=H_{d-1}+\ldots+H_1$.
Moreover,   $H_{d-k}\in W_{(d-1,1)}(V,\K)$ for all $k=1,\ldots,d-1$ since $\frac{(1)+\omega^k \s +(\o^k)^2\s^2+\ldots+(\o^k)^{d-1}\s^{d-1}}{d}\in\K[S_d]$. 

\bigskip
Now, we prove that each $H_{d-k}\in W_{(d-1,1)}^{d-k}(V,\K)$ for $k=1,\ldots,d-1$, that is, $\s H_{d-k}=\o^{d-k}H_{d-k}$ for all $k\in\{1,\ldots,d-1\}$. On the one hand,
\bigskip
$$\s H_{d-k}=\frac{\s F+\omega^k \s^2 F+(\o^k)^2\s^3 F+\ldots+(\o^k)^{d-2}\s^{d-1}F+(\o^k)^{d-1}F}{d}=$$ 
$$=\frac{\s F+\omega^k \s^2 F+\o^{2k}\s^3 F+\ldots+\o^{-2k}\s^{d-1}F+\o^{-k}F}{d}.$$
On the other hand,
$$\o^{d-k}H_{d-k}=\frac{\o^{d-k}F+\o^{d-k}\o^k \s F+ \o^{d-k} (\o^k)^2 \s^2 F+ \ldots+\o^{d-k}(\o^k)^{d-2}\s^{d-2}F+\o^{d-k}(\o^k)^{d-1}\s^{d-1}F}{d}=$$
$$=\frac{\o^{-k}F+\s F+ \o^{k} \s^2 F+\ldots+ \o^{-3k}\s^{d-2}F+\o^{-2k}\s^{d-1}F}{d}.$$
One can check that both expressions coincide, just as we wanted.

\bigskip
Finally, we are left to prove that  the sum of $W_{(d-1,1)}^1(V,\K),\ldots, W_{(d-1,1)}^{d-1}(V,\K)$ is direct. 
In order to prove it, we show that $$W_{(d-1,1)}^i(V,\K)\bigcap (\sum_{
    k\neq i, k=1 
}^{ d-1}W_{(d-1,1)}^k(V,\K))=0$$ 
for all $i=1,\ldots, d-1$.
We use induction reasoning on the number $n$ of summands.
For $n=0$ we have trivially that $W_{(d-1,1)}^i(V,\K)\cap (0)=0$. We assume that $W_{(d-1,1)}^i\bigcap (\sum_{
    k=i_1 
}^{ i_{n-1}}W_{(d-1,1)}^k(V,\K))=0$ with $i\notin\{i_1,\ldots,i_{n-1}\}$ and we prove  $W_{(d-1,1)}^i(V,\K)\bigcap (\sum_{
    k=i_1 
}^{ i_{n}}W_{(d-1,1)}^k(V,\K))=0$ with $i\notin\{i_1,\ldots,i_n\}$. Suppose that $G\in W_{(d-1,1)}^i(V,\K)$ and $G\in \sum_{
    k=i_1 
}^{ i_{n}}W_{(d-1,1)}^k(V,\K)$. So, 
\begin{equation}\label{eq2}
    G=H_{i_1}+\ldots+H_{i_n}
\end{equation}
with $H_{i_k}\in W_{(d-1,1)}^{i_k}(V,\K)$ for $k=1,\ldots,n$. Therefore,  
$$\s G=\s H_{i_1}+\ldots+\s H_{i_n}$$
and applying the definition of $W_{(d-1,1)}^k(V,\K)$ we obtain
\begin{equation}\label{eq3}
\o^{i}G=\o^{i_1}H_{i_1}+\ldots+\o^{i_n}H_{i_n}.
\end{equation}
\noindent
We subtract \eqref{eq3} 
from \eqref{eq2}
multiply by $\o^{i_n}$, obtaining as a result the following,
$$(\o^{i_n}-\o^{i})G=(\o^{i_n}-\o^{i_1})H_{i_1}+\ldots+(\o^{i_n}-\o^{i_{n-1}})H_{i_{n-1}}.$$
This implies that $(\o^{i_n}-\o^{i})G\in W_{(d-1,1)}^i(V,\K)\bigcap (\sum_{k=i_1}^{i_{n-1}}W_{(d-1,1)}^k(V,\K))=0$, so that $G=0$ since $\o^{i_n}-\o^{i}\neq 0$.
\end{proof}

\begin{remark}\rm
  Observe that we also have the decomposition $W_{(d-1,1)}(X,\K)=\bigoplus_{i=1}^{d-1}W_{(d-1,1)}^i(X,\K)$ for a general set $X$.  
\end{remark}

\begin{proposition}\label{prop:pregunta}
    Let $\K$ be an algebraically closed field and $F\colon V\times \overset{\text{(d)}}{\cdots} \times V\to \K$ be a multilinear form with $d\geq 3$ and satisfying $F(v,\ldots,v)=0$ for all $v\in V$. There exists $u\in V\setminus\{0\}$ such that $F(u,\ldots,u, V)=0$.
\end{proposition}
\begin{proof}
    Let $\{v_1,\ldots,v_n\}$ be a basis of $V$. We need to find $v=\sum_{j=1}^n \l_j v_j$ such that, for $i=1,\ldots,n$,
    $$G_i(\l_1,\ldots,\l_n):=F(\sum_{j=1}^n \l_j v_j,\ldots,\sum_{j=1}^n \l_j v_j,v_i)=0.$$
We are thus looking for non-trivial common solutions of the polynomials $G_1,\dots G_n\in \K[\l_1,\ldots,\l_n]_{d-1}$, and we will interpret this as finding a point in $\p^{n-1}$ in the intersection of the corresponding $n$ hypersurfaces. Observe that there is a relation
$$\sum_{i=1}^n \l_i G_i(\l_1,\ldots,\l_n)=F(\sum_{i=1}^n \l_i v_i,\ldots, \sum_{i=1}^n \l_i v_i)=0.$$ 
This means that the morphism $\OO_{\p^{n-1}}\to \OO_{\p^{n-1}}(d-1)^n$ given by $G_1,\dots,G_n$ is zero when composed with the morphism $\OO_{\p^{n-1}}(d-1)^n\to\OO_{\p^{n-1}}(d)^n$ provided by the Euler exact sequence. 
In other words, we have a factorization
    $$\xymatrix{\OO_{\p^{n-1}} \ar[dr] \ar[r] \ar@/_/[ddr]& \Omega_{\p^{n-1}}(d) \ar@{^{(}->}[d] \\ & \OO_{\p^{n-1}}(d-1)^n \ar[d] \\
    & \OO_{\p^{n-1}}(d)}$$
showing that the original morphism $\OO_{\p^{n-1}}\to \OO_{\p^{n-1}}(d-1)^n$ is injective if and only if $\OO_{\p^{n-1}}\to \Omega_{\p^{n-1}}(d)$ is injective. But this cannot be injective because, as it is well known (see also Remark \ref{remark:distribucion}), $c_{n-1}(\Omega_{\p^{n-1}}(d))>0$ when $d\ge3$. This completes the proof.
\end{proof}

\begin{theorem}\label{teo:main}
    Let $V$ be a vector space over a field $\K$ and fix $d\geq 3$. Suppose that the characteristic of $\K$ does not divide $d!$ and that $\K$ contains a  primitive $d^{\text{th}}$-root of unity, $\o$.
    If $F\in \left(\bigoplus_{\l\neq (d), (d-1,1)} W_\l(V,\K)\right)\oplus W_{(d-1,1)}^i(V,\K)$ for $i=1,\ldots,d-1$, then $\Det(F)=0$. 
\end{theorem}
\begin{proof}
Let us denote by $\Bar{\K}$ the algebraic closure of $K$, by $\Bar{V}$ the completion of $V$, that is, $\Bar{V}=V\otimes \Bar{K}$, and by $\Bar{F}$ the extension of $F\in \left(\bigoplus_{\l\neq (d), (d-1,1)} W_\l(V,\K)\right)\oplus W_{(d-1,1)}^i(V,\K)$. If we see $F$ and $\Bar{F}$ as hypermatrices, they are completely the same. So, if we prove that $\Det(\Bar{F})=0$ we are done.
By hypothesis, $\Bar{F}\in \left(\bigoplus_{\l\neq (d), (d-1,1)} W_\l(\Bar{V},\Bar{\K})\right)\oplus W_{(d-1,1)}^i(\Bar{V},\Bar{\K})$ and we can decompose it as $\Bar{F}=F'+F''$ with $F'\in \bigoplus_{\l\neq (d), (d-1,1)} W_\l(\Bar{V},\Bar{\K})$ and $F''\in W_{(d-1,1)}^i(\Bar{V},\Bar{\K})$.
On the one hand, since $F''(v,\ldots,v)=0$ for all $v\in \Bar{V}$, by Proposition \ref{prop:pregunta}, there exists $u\in \Bar{V}\setminus\{0\}$ such that $F''(u,\ldots,u,\Bar{V})=0$. Moreover, since $F''\in W_{(d-1,1)}^i(\Bar{V},\K)$, by Proposition \ref{prop:descomp},  we know that $\s F''=\o^{d-i} F''$ where $\s$ is a $d$-cycle and $\o$. So, $F''$ satisfies condition \eqref{estrella} for $u$. 
On the other hand, by Remark \ref{remark:any},  $F'$ satisfies condition \eqref{estrella} for all $v\in \Bar{V}$. In particular, $F'$ satisfies condition \eqref{estrella} for $u\in \Bar{V}\setminus\{0\}$.
Therefore, $\Bar{F}$ satisfies condition \eqref{estrella} for such $u$. So, $\Det(\Bar{F})=0$.
\end{proof}

\begin{remark}\rm
Let us recall that for the symmetric case, $F\in W_{(d)}(V,\K)$, L. Oeding described in \cite{OEDING} all the irreducible factors of $\Det(F)$. 
Note now that if $F\in W_{(d-1,1)}(V,\K)$ then $\Det(F)$ is also not necessarily zero. For example, if $n=2$ and $d=3$ we have the explicit formula of the hyperdeterminant (see \cite[Proposition 1.7, Chapter 14]{Kapranov} for the original outcome and \cite[Example 5.6]{OttHyp} for the reformulation we use). So, consider the multilinear map $F\colon \K^2\times \K^2\times \K^2\to \K$ given by $F((x_1,y_1),(x_2,y_2),(x_3,y_3))=x_1 x_2 y_3-x_1 y_2 x_3+y_1 y_2 x_3-x_1 y_2 y_3$ whose associated multidimensional matrix is $A=(a_{ijk})$ with $i,j,k=1,2$ and $a_{111}=0,a_{222}=0,a_{112}=1,a_{211}=0,a_{121}=-1,a_{221}=1,a_{122}=-1,a_{212}=0$. Observe that $F\in W_{(2,1)}(\K^2,\K)$ since $F+(123)F+(132)F=0$, that is,
$F((x_1,y_1),(x_2,y_2),(x_3,y_3))+F((x_2,y_2),(x_3,y_3),(x_1,y_1))+F((x_3,y_3),(x_1,y_1),(x_2,y_2))=0$.
Its hyperdeterminant is:
$$\Det(F)= \left(\begin{vmatrix}
a_{111} & a_{122}\\
a_{211} & a_{222}
\end{vmatrix}+ \begin{vmatrix}
    a_{121} & a_{112}\\
    a_{221} & a_{212}
\end{vmatrix}\right)^2-4\begin{vmatrix}
    a_{111} & a_{112}\\
    a_{211} & a_{212}
\end{vmatrix}\cdot \begin{vmatrix}
    a_{121} & a_{122}\\
    a_{221} & a_{222}
\end{vmatrix}=$$

$$=\left(\begin{vmatrix}
0 & -1\\
0 & 0
\end{vmatrix}+ \begin{vmatrix}
    -1 & 1\\
    1 & 0
\end{vmatrix}\right)^2-4\begin{vmatrix}
    0 & 1\\
    0 & 0
\end{vmatrix}\cdot \begin{vmatrix}
    -1 & -1\\
    1 & 0
\end{vmatrix}=1\neq 0.$$

\end{remark}

\begin{remark}\label{remark:distribucion}\rm
If we allow $d=2$ in the proof of Proposition \ref{prop:pregunta}, we have that $c_{n-1}(\Omega_{\p^{n-1}}(2))=1$ if $n$ is odd even and $c_{n-1}(\Omega_{\p^{n-1}}(2))=0$ if $n$ is even, so that in this last case we can have an injective morphism. This corresponds to the fact that, for $n\times n$ matrices, the determinant is always zero if $n$ is odd but can be nonzero if $n$ is even. It is worth to notice that we could have rephrased everything in the language of distributions. Indeed giving polynomials $G_1,\dots,G_n$ as in our proof is equivalent to give a distribution of degree $d-2$ in $\p^{n-1}$, i.e a section of  $\Omega_{\p^{n-1}}(d)$. And a well-known result in this language (see for example \cite[Proposition 2.1] {distribuciones}) states that a distribution of degree $d-2$ in $\p^{n-1}$ has singular locus except possibly when $d=2$ and $n$ even. It is worth to mention that 
$$H^0(\Omega_{\p^{n-1}}(d))=\Sc_{(d-1,1)} H^0(\OO_{\p^{n-1}}(1)),$$ 
and this is isomorphic to any of the spaces $W_{(d-1,1)}^i(X,\K)$. In fact, and more generally, from Bott's formula we have an isomorphism $$H^0(\Omega^p_{\p^{n-1}}(d))=\Sc_{(d-p,1,\dots,1)} H^0(\OO_{\p^{n-1}}(1)).$$
This can be obtained directly by induction on $p$ (the case $p=0$ being trivial) and using the exact sequence
$$0\to H^0(\Omega^p_{\p^{n-1}}(d))\to \bigwedge^{p+1}H^0(\OO_{\p^{n-1}}(1))\otimes H^0(\OO_{\p^{n-1}}(d-p-1))\to H^0(\Omega^{p-1}_{\p^{n-1}}(d))\to0$$
(taking cohomology in the short exact sequences of the long Euler exact sequence). 

\end{remark}

\begin{remark}\label{remark:subspaces}\rm

From the above remark we easily get the known result that $H^0(\Omega_{\p^{n-1}}(d))$ has dimension equal to $(d-1)\binom{n+d-2}{d}$ so that this is also the dimension of any $\dim W^i_{(d-1,1)}(V,\K)$. In particular, our Theorem \ref{teo:main} is saying that, inside the $n^d$-dimensional space $V^*\otimes\overset{\text(d)}{\ldots}\otimes V^*$, the hypersurface determined by $\Det$ contains $d-1$ linear subspaces of codimension $\binom{n+d-1}{d}+(d-2)(d-1)\binom{n+d-2}{d}$. They meet in the common subspace of codimension $\binom{n+d-1}{d}+(d-1)^2\binom{n+d-2}{d}$ found in \cite{TocinoHyper} and \cite{repre}.

\end{remark}

\bibliographystyle{plain}
\bibliography{ref}

\begin{thebibliography}{10}

\bibitem{repre}
E.~Arrondo.
\newblock Representation theory of finite groups through (basic) algebraic geometry.
\newblock {\em arXiv preprint}, 2021.

\bibitem{fulton2013representation}
W.~Fulton and J.~Harris.
\newblock {\em Representation theory: a first course}.
\newblock Springer-Verlag, 1991.

\bibitem{Kapranov}
I.~M. Gelfand, M.~M. Kapranov, and A.~V. Zelevinsky.
\newblock {\em Discriminants, resultants, and multidimensional determinants}.
\newblock Mathematics: Theory \& Applications. Birkh\"{a}user Boston, Inc., Boston, MA, 1994.

\bibitem{Isaacs}
I.~M. Isaacs.
\newblock {\em Character Theory of Finite Groups}.
\newblock Academic Press, Inc., 1976.

\bibitem{distribuciones}
J.-P. Jouanolou.
\newblock {\em Equations de Pfaff algebriques sur un espace projectif}.
\newblock Lecture Notes in Mathematics, 708, lecture 2. Springer, 1979.

\bibitem{MetropolisRota3}
N.~Metropolis and Gian-Carlo Rota.
\newblock Symmetry classes: functions of three variables.
\newblock {\em Amer. Math. Monthly}, 98(4):328--332, 1991.

\bibitem{RotaMetropolisStein}
N.~Metropolis, Gian-Carlo Rota, and Joel~A. Stein.
\newblock Theory of symmetry classes.
\newblock {\em Proc. Nat. Acad. Sci. U.S.A.}, 88(19):8415--8419, 1991.

\bibitem{OEDING}
Luke Oeding.
\newblock Hyperdeterminants of polynomials.
\newblock {\em Advances in Mathematics}, 231(3):1308--1326, 2012.

\bibitem{OttHyp}
Giorgio Ottaviani.
\newblock Introduction to the hyperdeterminant and to the rank of multidimensional matrices.
\newblock In {\em Commutative algebra}, pages 609--638. Springer, New York, 2013.

\bibitem{TocinoHyper}
A.~Tocino.
\newblock The hyperdeterminant vanishes on all but two {S}chur functors.
\newblock {\em J. Algebra}, 450:316--322, 2016.

\end{thebibliography}
\end{document}